\documentclass[11pt,letterpaper,oneside]{article}
\usepackage{amsmath,amsthm,amssymb,amscd,abbrViro,graphicx,xcolor}
\begin{document} 
\title{Characterization of affine links in the projective space}
\author{Oleg Viro}
\date{}
\maketitle

\section{Introduction}\label{s0}
A {\em link} in the real projective space $\rpt$ is a smooth closed 
1-dimensional submanifold of $\rpt$. As usual, if a link is connected, 
then it is called a {\em knot\/}. A link $L\subset \rpt$ is said to be 
{\em affine} if
it is isotopic to a link, which does not intersect some plane 
$\rpp\subset \rpt$.

The main theorem of this paper  provides necessary and sufficient condition
for a link in $\rpt$ to be affine:

\begin{MTh}\label{MT}
A link $L\subset \rpt$ is affine if and only if 
$\pi_1(\rpt\sminus L)$ contains a non-trivial element of order 2. 
\end{MTh}

\subsection{Known results}\label{s0.1}
The problem of determining whether a link is affine was considered in
literature and there are results in this direction, mostly about necessary
conditions:

{\bf Homology condition.}  Each connected component of a link 
$L\subset\rpt$ realizes a homology class, an element of 
$H_1(\rpt;\ZZ)=\ZZ$.   
All components of an affine link $K\subset\rpt$ realize  
$0\in H_1(\rpt;\ZZ)$.  The converse is not true: there exist knots 
homological to zero in $\rpt$, which are not affine. 
A few examples are shown in Figure \ref{f1}.

{\bf Self-linking number.} If a knot $K\subset\rpt$ realizes 
$0\in H_1(\rpt;\ZZ)$, then a self-linking number $\sl(K)\in\Z$ is 
defined as the linking number modulo 2  of the connected 
components of the preimage $\widetilde K\subset S^3$ of $K$ under the 
covering $S^3\to\rpt$, see \cite{D1}, \S7. 

If $K$ is affine, then $\sl(K)=0$, see \cite{D1}, \S7. 
For the knot $2_1$ shown in Figure \ref{f1}, this invariant equals 2, 
and this is why $2_1$ is not affine.

{\bf Exponents of monomials in the bracket polynomial.} 
If a knot $K\subset \rpt$ is affine, 
then the exponents of all monomials of its bracket polynomial $V_K$ defined 
by Drobotukhina in \cite{D1} are congruent to each other modulo 4. 

\begin{figure}[t]
\centerline{\includegraphics[scale=.75]{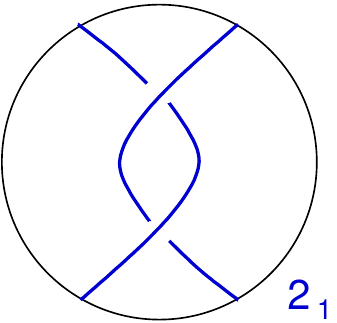}\hspace*{1cm}
\includegraphics[scale=.75]{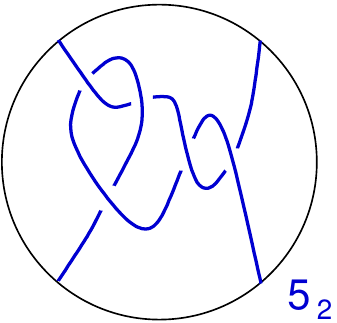}\hspace*{1cm}
\includegraphics[scale=.75]{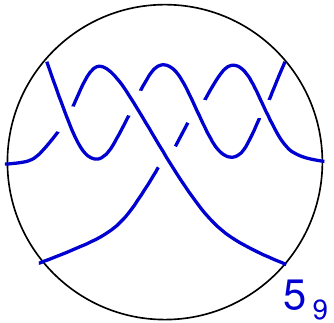}}
\caption{}
\label{f1}
\end{figure} 
The self-linking and exponent conditions are independent: 
$\sl5_2=0$, while $V_{5_2}=A^4+A^2-1-2A^{-2}+A^{-4}+2A^{-6}-2A^{-10}+A^{-14}$, on the other hand, $V_{5_9}=A^{-8}+A^{-12}-A^{-20}$, while $\sl5_9=3$. 
See Figure \ref{f1} and Drobotukhina's table \cite{D2}.  

\subsection{Comparison to Main Theorem}\label{s0.2}
The main theorem of this note provides necessary and sufficient condition
for a link in $\rpt$ to be affine. It is similar to the famous results of
the classical knot theory, like the Dehn-Papakyriacopoulos characterization
of the unknot as the only knot  $K\subset\R^3$ with $\pi_1(\R^3\sminus
K)$ isomorphic to $\Z$.    

However conditions formulated in terms of fundamental groups are not easy
to check. Therefore the known necessary conditions mentioned above may
happen to be more useful for checking if a specific knot is affine. 

\subsection{Reformulation of Main 
Theorem\\ \hspace*{4.4cm} and its conjectural{ generalization}}\label{s0.3}
The property of a projective link of being affine admits the 
following obvious reformulation:\\
{\it A link $L\subset \rpt$ is affine if and only if there exists an
embedding $i:D^3\to\rpt$ such that $L\subset i(D^3)$.} 

This property of links in $\rpt$ admits the following generalization to
links in an arbitrary 3-manifold. A link $L$ in a 3-manifold $M$ is called 
{\sfit localizable\/} if there exists an embedding $i:D^3\to M$ of the ball
$D^3$ such that $L\subset i(D^3)$.

It would be interesting to find a characterization of localizable links. 
The following conjecture would provide a generalization of Main Theorem 
to closed 3-manifolds with trivial $\pi_2$.\medskip

\noindent{\bf Conjecture.} {\it Let $M$ be a closed 3-manifold with
$\pi_2(M)=0$\footnote{As follows from the Papakyriacopoulos sphere theorem, 
this
condition can be reformulated as non-existence of a 2-sphere $\GS$ 
embedded into 
the orientation covering space of $M$ in such a way that $\GS$ does not bound 
a 3-ball there.}. 
Then a link $L\subset M$ is localizable if and only if $\pi_1(M\sminus L)$
contains a subgroup $G$ which is mapped isomorphically onto 
$\pi_1(M)$ by the inclusion homomorphism $\pi_1(M\sminus L)\to\pi_1(M)$.}
\medskip

\section{Proofs}\label{s1}

\subsection{Proof of Main Theorem. Necessity.}\label{s1.1} 
Assume that $L$ is an affine link. Since 
the group $\pi_1(\rpt\sminus L)$ is invariant under isotopy, it suffices
to prove that $\pi_1(\rpt\sminus L)$ contains an element of order 2 if 
$L$ does not intersect a plane $\rpp\subset\rpt$.

If $L\cap\rpp=\empt$, then $L$ is contained in the affine part 
$\rpt\sminus\rpp=\R^3$ of $\rpt$,
and  $\rpt\sminus L$ is a union of $\rpt\sminus(L\cup\rpp)$ and a regular 
neighborhood of $\rpp$. The van Kampen Theorem applied to this presentation 
of $\rpt\sminus L$ implies that  $\pi_1(\rpt\sminus L)$ is the free product
$\ZZ*\pi_1(\R^3\sminus L)$. Hence it contains a  non-trivial element of
order 2.\hfill\qed\bigskip

\subsection{A lemma about a map of the projective plane}\label{s1.2}
The proof of sufficiency is prefaced with the following simple 
homotopy-theoretic lemma: 

\begin{lem}\label{lem1} Let $f:\rpp\to\rpp$ be a map inducing isomorphism
on fundamental group. Then the covering map $\widetilde f:S^2\to S^2$ is
not null homotopic.
\end{lem}

\begin{proof} Consider the diagram
$$
\minCDarrowwidth8pt
\begin{CD}
0@>>>H_2(\rpp;\ZZ)@>>>H_2(S^2;\ZZ)@>>>H_2(\rpp;\ZZ)@>>>H_1(\rpp;\ZZ)@>>>0\\
@. @VV{f_*}V @VV{\widetilde f_*}V @VV{f_*}V  @VV{f_*}V @. \\
0@>>>H_2(\rpp;\ZZ)@>>>H_2(S^2;\ZZ)@>>>H_2(\rpp;\ZZ)@>>>H_1(\rpp;\ZZ)@>>>0
\end{CD}
$$ 
in which the rows are segments of the Smith sequence (see, e.g., \cite{Bre}) 
for the antipodal involution on $s:S^2\to S^2:x\mapsto-x$. The diagram 
is commutative, because $\widetilde f$ commutes with $s$. 
Notice that all the groups in this diagram are isomorphic to $\ZZ$. 
Exactness of the Smith sequences implies that the middle horizontal 
arrows in both rows are trivial, and the horizontal arrows next to them
are isomorphism. By assumption, the rightmost vertical arrow is an
isomorphism. Therefore the next vertical arrow is an isomorphism. This
isomorphism 
coincides with the homomorphism represented by leftmost vertical arrow.
Hence, the next arrow $\widetilde f_*:H_2(S^2;\ZZ)\to H_2(S^2;\ZZ)$ is 
an isomorphism. Thus $\widetilde f$ is not null homotopic.    
\end{proof}

\subsection{Proof of Main Theorem: Sufficiency}\label{s1.3}
Assume that $\pi_1(\rpt\sminus L)$ contains a
non-trivial element $\Gl$ of order two. Realize $\Gl$ by a smoothly 
embedded loop
$l:S^1\to \rpt\sminus L$. 
\subsubsection{From a loop of order 2 to a singular projective
plane in $\rpt\sminus L$}\label{s1.3.1}
Since $\Gl^2=1$, there exists a
continuous map $D^2\to \rpt\sminus L$ such that its restriction to the
boundary circle $\p D^2$ is the square of $l$. Together with $l$, this 
map gives a continuous map of the projective plane 
$P=D^2/_{x\sim -x, \text{ for } x\in\p D^2}$ to $\rpt\sminus L$. 
%
Denote this map by $g$. 
So, $g:P\to\rpt\sminus L$ is a generic differentiable map 
which induces
a monomorphism $g_*$ of $\pi_1(P)=\ZZ$ to $\pi_1(\rpt\sminus L)$. 

\subsubsection{Singular sphere over the singular projective plane}\label{s1.3.2}  
Let $p:S^3\to\rpt$ be the canonical two-fold covering. Consider its
restriction $S^3\sminus p^{-1}(L)\to\rpt\sminus L$. 
Observe that $\Gl$ does not belong to the group of this covering, because 
otherwise $\pi_1(S^3\sminus p^{-1}(L))$ would contain a non-trivial element 
of order two, which is impossible - a link group does not have any
non-trivial element of finite order. 

Therefore the covering of the projective plane induced from $p$ via $g$
is a non-trivial two-fold covering. Its total space is a 2-sphere. Denote
it by $S$. The map $\widetilde g$ which covers $g$ maps $S$ to    
$S^3\sminus p^{-1}(L)$. 

\subsubsection{Non-contractibility in $S^3\sminus p^{-1}(L)$ of the
singular sphere }\label{s1.3.3}
Let us choose a point $x\in L$. 
%
%
Its complement $\rpt\sminus\{x\}$ is homotopy equivalent to $\rpp$. 
Indeed, the projection from $x$ to any projective plane, which does not 
contain $x$, is a deformation retraction  $\rpt\sminus\{x\}\to\rpp$. 
The composition of
$g:P\to\rpt\sminus L$ with the inclusion $\rpt\sminus L\to \rpt\sminus\{x\}$
induces an isomorphism $\pi_1(P)\to\pi_1(\rpt\sminus\{x\})$. Both spaces, 
$P$ and $\rpt\sminus\{x\}$, have homotopy type of $\rpp$. 
Lemma \ref{lem1} implies that $\widetilde g:S\to S^3\sminus p^{-1}(x)$ 
is not null-homotopic. 

\subsubsection{Existence of a non-singular sphere $\GS_0\subset S^3$ splitting $p^{-1}L$}\label{s1.3.4}
Denote by $\Gs$ the antipodal involution $S^3\to S^3:x\mapsto -x$.

\begin{lem}\label{lem2} There exists a non-singular polyhedral 
submanifold $\GS_0$ of $S^3$
homeomorphic to $S^2$ such that $\GS_0\cap p^{-1}(L)=\empt$ and the two points
of $p^{-1}(x)$ belong to different connected components of
$S^3\sminus\GS_0$.
\end{lem}

\begin{proof} First, let us apply Whitehead's modification \cite{Whi} of 
the Papakyriakopoulos Sphere Theorem \cite{Papa}. 

Recall the statement of this theorem (Theorem (1.1) of \cite{Whi}): 
{\it For any connected,
orientable triangulated 3-manifold $M$ and subgroup $\GL\subset\pi_2(M)$  
which is invariant under the action of $\pi_1(M)$, if $\GL\ne\pi_2(M)$, 
then $M$ contains a non-singular polyhedral 2-sphere which is essential 
$\mod\GL$.} 

We will apply this theorem to $$M=S^3\sminus p^{-1}(L), \; \; 
\GL=\Ker\left(\ink_*:\pi_2(S^3\sminus p^{-1}(L))\to\pi_2(S^3\sminus
p^{-1}(x))\right).$$
We know that $\pi_2(S^3\sminus\{x\})=\pi_2(S^2)=\Z$ and that 
the homotopy class of $\widetilde g$ is non-trivial in 
$\pi_2(S^3\sminus p^{-1}(x))$. Therefore the homotopy class of $\widetilde
g$ does
not belong to $\GL$, and hence $\GL\ne\pi_2(S^3\sminus p^{-1}(L))$. Thus, 
all the assumptions of the Whitehead theorem are fulfilled.  

Let us denote by $\GS_0$  a non-singular polyhedral 2-sphere whose
existence is stated by the Whitehead theorem. By the Alexander Theorem
\cite{Alx}, $\GS_0$ bounds in $S^3$ two domains homeomorphic to ball. 
Since $\GS_0$ is not null homotopic in $S^3\sminus
p^{-1}(x)$, each of these domains contains a point of $p^{-1}(x)$.
\end{proof}

\subsubsection{Improving the splitting sphere }\label{s1.3.5}
\begin{lem}\label{lem3} There exists a smooth submanifold $\GS$ of $S^3$
homeomorphic to $S^2$ such that $\GS\cap p^{-1}(L)=\empt$, the two points
of $p^{-1}(x)$ belong to different connected components of
$S^3\sminus\GS$,
and either $\GS=\Gs(\GS)$ or $\GS\cap\Gs(\GS)=\empt$. 
\end{lem}

\begin{proof}
Any polyhedral compact surface can be approximated by a smooth
2-submanifold. Let  $\GS_1$ be a smooth submanifold of $S^3$ approximating
$\GS_0$ in $S^3\sminus p^{-1}(L)$. 

The antipodal involution $\Gs:S^3\to S^3$ is an
automorphism of the covering $p:S^3\to\rpt$, therefore $p^{-1}(L)$ is
invariant under $\Gs$ and $\Gs(\GS_1)\subset S^3\sminus p^{-1}(L)$.

Let us assume that $\GS_1$ and $\Gs(\GS_1)$ are transversal -- this can be
achieved by an arbitrarily small isotopy of $\GS_1$. 
Then the intersection
$\GS_1\cap\Gs(\GS_1)$ consists of disjoint circles. 


Take a connected component $C$ of $\GS_1\cap\Gs(\GS_1)$ which is innermost in
$\Gs(\GS_1)$ 
(i.e., which  bounds in $\Gs(\GS_1)$ a disc $D$ containing no other 
components of $\GS_1\cap\Gs(\GS_1)$). 

First, assume that $C\ne \Gs(C)$. In this case, make surgery on $\GS_1$ 
along $D$:
remove a regular neighborhood $N$ of $C$ from $\GS_1$ and attach to $\p N$ 
two discs parallel to $D$. This surgery does not change the homology class 
with coefficients in $\ZZ$ realized by $\GS_1$ in $S^3\sminus p^{-1}(x)$.

Denote by $\GS_2$ the result of this surgery on $\GS_1$. This is a disjoint 
union of two spheres. 
The sum of the homology classes realized by them is the same 
non-trivial element
of $H_2(S^2\sminus p^{-1}(x);\ZZ)=\ZZ$ which was realized by $\GS_1$. 
Therefore, one of the summands is
non-trivial. The corresponding component of $\GS_2$ separates the two points 
of $p^{-1}(x)$. Denote this component by $\GS_3$.
 Since $C\ne\Gs(C)$, the number of connected components of
$\GS_3\cap\Gs(\GS_3)$ is less than the number of connected components of 
$\GS_1\cap\Gs(\GS_1)$, all other properties of $\GS_1$ are inherited by
$\GS_3$, and we are ready to continue with the next connected component
of $\GS_3\cap\Gs(\GS_3)$ which bounds in $\Gs(\GS_3)$ a disc containing 
no other components of $\GS_3\cap\Gs(\GS_3)$).

Second, consider the case $C=\Gs(C)$.  Then the disc $D\subset\Gs(\GS_1)$
together with it's image $\Gs(D)\subset \GS_1$ form an embedded sphere, which
is invariant under $\Gs$ and does not meet the rest of $\GS\cup\Gs(\GS_1)$
besides along $C$, that is
$(D\cup\Gs(D))\cap((\GS_1\sminus\Gs(D)\cup(\Gs(\GS_1)\sminus D))=\empt $. 

If $D\cup\Gs(D)$ separates points of $p^{-1}(x)$, we are done: 
we can smoothen the corner of  $D\cup\Gs(D)$ along $C$ keeping it invariant
under $\Gs$ and take the result for $\GS$. 

If $D\cup\Gs(D)$ does not separate points of $p^{-1}(x)$, then 
$D\cup(\GS_1\sminus\Gs(D))$  separates points of $p^{-1}(x)$ 
(as well as its image under $\Gs$, that is 
$\Gs(D)\cup(\Gs(\GS)\sminus D)$). Indeed, the homology classes realized by 
$D\cup\Gs(D)$ and $D\cup(\GS_1\sminus\Gs(D))$ in $S^3\sminus p^{-1}(x)$ 
differ from each other by the homology class of
$\Gs(D)\cup(\GS_1\sminus\Gs(D))=\GS_1$ which is known to be nontrivial. 
So, if the class of  $D\cup\Gs(D)$ is trivial, then the class of 
$\Gs(D)\cup(\Gs(\GS_1)\sminus D)$ is not. 
Then smoothing of a corner along $C$ turns $D\cup(\GS_1\sminus\Gs(D))$ 
into a new sphere $\GS_2$ such that
$\GS_2\cap\Gs(\GS_2)$ has less connected components than  $\GS_1\cap\Gs(\GS_1)$. 

By repeating this construction, we will eventually build up a sphere 
$\GS\subset S^3\sminus p^{-1}(L)$ with the required properties.
\end{proof}

\subsubsection{Completion of the proof}\label{s1.3.6}
Let us return to the proof of Main Theorem. If the sphere $\GS$
provided by Lemma \ref{lem3} is invariant under $\Gs$, then $\GS$ divides
$S^3$ into two balls which are mapped by $\Gs$ homeomorphically to each
other. Let $B$ be one of them. The part of $p^{-1}(L)$ contained in 
$B$ can be moved by an isotopy fixed on a neighborhood of the boundary 
of $B$ inside an arbitrarily small metric ball in $S^3$. Using $\Gs$,
extend this isotopy to a $\Gs$-equivariant isotopy of the whole $S^3$.
The equivariant isotopy defines an isotopy of $\rpt$ which moves $L$
to a link contained in a small metric ball. This proves that $L$ is
an affine link.

Consider now the case in which the sphere $\GS$ provided by Lemma 
\ref{lem3} is not $\Gs$-invariant, but rather is disjoint from its image
$\Gs(\GS)$. Then spheres $\GS$ and $\Gs(\GS)$ divide $S^3$ into three
domains: two of them are balls bounded by $\GS$ and $\Gs(\GS)$,
respectively. Let us denote by $B$ the ball bounded by $\GS$, then
its image $\Gs(B)$ is bounded by $\Gs(\GS)$. Denote the third domain
by $E$. It is invariant under $\Gs$. 

If one of the points from $p^{-1}(x)$ belonged to $E$, then the other one
also would belong to $E$, and then the sphere $\GS$ would be contractible
in $S^3\sminus p^{-1}(x)$. Therefore $B\cap p^{-1}(L)\ne\empt$. Denote 
$B\cap p^{-1}(L)$ by $K$. This is a sublink of $p^{-1}(L)$.
It  can be moved by an isotopy fixed on a neighborhood of the boundary 
of $B$ inside an arbitrarily small metric ball in $S^3$. Then this isotopy 
can be extended to $\Gs$-equivariant isotopy of $S^3$ fixed on $E$.
This equivariant isotopy defines an isotopy of $\rpt$ which moves $p(K)$
to a link contained in a small metric ball. 

Thus our link $L$ is presented as a disjoint sum of an affine link $p(K)$
and the rest $L\sminus p(K)$ of $L$. 
If $L\sminus p(K)=\empt$, then we are done. If not,
then $\pi_1(S^3\sminus L)$ is presented as a free product of 
$\pi_1(B\sminus K)$ and $\pi_1(\rpt\sminus (L\sminus p(K))$. By the
assumption, the group $\pi_1(S^3\sminus L)$ has a non-trivial element of
order 2. The first factor, $\pi_1(B\sminus K)$ cannot contain such an
element, because this is a group of a classical link. Hence, the second 
factor,  $\pi_1(\rpt\sminus (L\sminus p(K))$, contains it, and we can
apply the constructions and arguments above to the link $L\sminus p(K)$. 
This link contains less
components than the original one, therefore, after several iterations, 
we will come to the 
situation in which $p^{-1}(L)\cap B=\empt$ 
\qed

\end{document}